\documentclass[10pt]{amsart}

\usepackage[all]{xy}  
\usepackage{version}
\excludeversion{suppress}

\numberwithin{equation}{section}

\newcommand{\ben}{\begin{enumerate}}
\newcommand{\een}{\end{enumerate}}

\newcommand{\bea}{\begin{eqnarray}}
\newcommand{\ba}{\begin{array}}
\newcommand{\bean}{\begin{eqnarray*}}
\newcommand{\ea}{\end{array}}
\newcommand{\eea}{\end{eqnarray}}
\newcommand{\eean}{\end{eqnarray*}}
\newcommand{\beq}{\begin{equation}}
\newcommand{\eeq}{\end{equation}}
\newcommand{\bthm}{\begin{thm}}
\newcommand{\ethm}{\end{thm}}
\newcommand{\blem}{\begin{lem}}
\newcommand{\elem}{\end{lem}}
\newcommand{\bprop}{\begin{prop}}
\newcommand{\eprop}{\end{prop}}
\newcommand{\bcor}{\begin{cor}}
\newcommand{\ecor}{\end{cor}}
\newcommand{\bdfn}{\begin{dfn}}
\newcommand{\edfn}{\end{dfn}}
\newcommand{\brem}{\begin{rem}}
\newcommand{\erem}{\end{rem}}
\newcommand{\bpf}{\begin{proof}}
\newcommand{\epf}{\end{proof}}
\newcommand{\bfact}{\begin{fact}}
\newcommand{\efact}{\end{fact}}
\newcommand{\bobs}{\begin{obs}}
\newcommand{\eobs}{\end{obs}}
\newcommand{\bexam}{\begin{exam}}
\newcommand{\eexam}{\end{exam}}

\alph{enumii} \roman{enumiii}

\newtheorem{thm}{Theorem}[section]
\newtheorem{prop}[thm]{Proposition}
\newtheorem{lem}[thm]{Lemma}

\newtheorem{cor}[thm]{Corollary}
\newtheorem{dfn}[thm]{Definition}
\newtheorem{rem}[thm]{Remark}
\newtheorem{fact}[thm]{Fact}
\newtheorem{claim}[thm]{Claim}
\newtheorem{obs}[thm]{Observation}
\newtheorem{exam}[thm]{Example}

 \newtheoremstyle{claimstyle}%
   {}
   {}
   {\normalfont}
   {}
   {\itshape}
   {.}
   { }
   {\thmnote{#3}}

\theoremstyle{claimstyle}

\def\cA{\mathcal A}                    \def\cC{\mathcal C}
             \def\cF{\mathcal F}       
\def\cL{{\mathcal L}}           \def\cM{\mathcal M}        \def\cP{{\mathcal P}}
             \def\cV{\mathcal V}       \def\cJ{\mathcal J}   \newcommand{\J}{\mathcal{J}}
\def\cS{\mathcal S}             \def\cE{\mathcal E}
           \def\cK{\mathcal K}

\def\N{{\mathbb N}}                      
\def\C{{\mathbb C}}                      \def\oc{{\hat \C}}

\newcommand{\cbar}{\hat{{\mathbb C}} }

                             \def\d{\delta}
                          
\def\g{\gamma}                           \def\l{\lambda} \def\la{\lambda}
                         
               \def\sg{\sigma}

\newcommand{\ep}{\varepsilon}
\newcommand{\ph}{\varphi}
\newcommand{\al}{\alpha}
\newcommand{\ga}{\gamma}

\def\1{1\!\!1}

\def\and{\text{ and }}

\def\F{{\mathcal F}}

     \def\HD{\text{{\rm HD}}}   
         \def\PD{\text{\rm PD}}

\def\({\bigl(}                \def\){\bigr)}

                        \def\^{\tilde}

\def\D{{\mathbb D}}

\newcommand{\pfl}{\cL _{{\l },t}}
\newcommand{\pfj}{\cL _{\sg^j({\l }),t}}
\newcommand{\pfjj}{\cL _{\sg^{j+1}({\l }),t}}

\def\shift{\sigma}

\def\lpt{{\underline P}_{\l}(t)}
\def\upt{\overline P_\l(t)}
\def\lp{{\underline P}_{\l}}
\def\up{\overline P_\l}

\def\set{\Lambda}
\def\basept{\eta}
\def\hbasept{\underline h_\basept}

\begin{document}

\title[Regularity (and irregularity) of fiber dimensions of RDS]
{ \bf\large Regularity and irregularity of fiber dimensions of non-autonomous dynamical systems}
\date{\today}

\author[\sc Volker MAYER]{\sc Volker MAYER}
\address{Volker Mayer, Universit\'e de Lille I, UFR de Math\'ematiques,
UMR 8524 du CNRS,
59655 Villeneuve d'Ascq Cedex, France}
\email{volker.mayer@math.univ-lille1.fr\newline \hspace*{0.3cm} Web:
math.univ-lille1.fr/$\sim$mayer}

\author{Bart{\l}omiej Skorulski} \address{Departamento de
  Matem{\'a}ticas, Universidad Cat{\'o}lica del Norte, Avenida Angamos
  0610, Antofagasta, Chile} \email{bskorulski@ucn.cl}

\author{Mariusz Urba\'nski}
\address{Department of
  Mathematics, University of North Texas, Denton, TX 76203-1430, USA}
\email{urbanski@unt.edu \newline \hspace*{0.42cm} \it Web: \rm
  www.math.unt.edu/$\sim$urbanski}

%
%

\begin{abstract}
This note concerns non-autonomous dynamics of rational functions and, more
precisely, the fractal behavior of the Julia sets under perturbation of non-autonomous 
systems. We provide a necessary and sufficient condition for holomorphic stability
which leads to H\"older continuity of dimensions of hyperbolic non-autonomous Julia sets
with respect to the $l^\infty$-topology on the parameter space.
On the other hand we show that, for some particular family, the Hausdorff and packing
dimension functions are not differentiable at any point and that these dimensions 
are not equal on an open dense set of the parameter space still with respect to the $l^\infty$-topology.
\end{abstract}

\thanks{The research of the third named author was supported in part by the NSF Grant DMS 1001874.}

\maketitle

\section{Introduction}

Let $\cF=\big\{ f_{\tau} \,;\; \tau\in \set_0 \big\}$ be a holomorphic family of rational functions depending 
analytically on a parameter $\tau\in \set_0 $, $\set_0$ being some open and connected
 subset of $\C^{d}$, $d \geq 2$.
We investigate the dynamics of functions
$$f_{\l_n}\circ f_{\l_{n-1}}\circ ...\circ f_{\l_1} \;\; , \quad n\geq 1\, ,$$
where each $f_{\l_j}$ is an arbitrarily chosen function of the family $\cF$.
Such a dynamical system is usually called \emph{non-autonomous}. They generalize
\emph{deterministic dynamics} (where all the functions
$f_{\l_j}$ equal one fixed rational map) and \emph{random dynamics} (where the functions
$f_{\l_j}$ are chosen according to some probability law) that first have been
considered by Fornaess and Sibony \cite{FS91}.
If $\l=(\l_1,\l_2,...)\in \set_0^\N$ then it is convenient to denote
 $$f_\l^n=f_{\l_n}\circ f_{\l_{n-1}}\circ ...\circ f_{\l_1}\,.$$
Like in deterministic dynamics, the normal family behavior of $(f_\l^n)_n$
splits the sphere into two subsets. The Fatou set $\F_\l$, i.e. the set of points 
for which $(f_\l^n)_n$ is normal on some neighborhood, and its complement the Julia set $\J_\l$.
We are going to investigate the fractal nature of the Julia set $\J_\l$ and, more precisely, the dependence 
of the fractal dimensions of  $\J_\l$ on the parameter $\l\in \set_0^\N$.

The deterministic hyperbolic case is completely understood by now. 
Indeed in 1979, R. Bowen \cite{Bow79} showed that the Hausdorff dimension of the Julia set 
can be expressed by the zero of a pressure function. 
The picture was completed by D. Ruelle \cite{Ru82} who showed that this dimension depends real
analytically on the function. More recently, random dynamics became an active area
and both Bowen's formula and Ruelle's real analyticity result have its counterparts
in random dynamics. Bowen's formula has been established for 
various random dynamical systems (see e.g. \cite{msu11} and the corresponding
references in this monograph) and H. Rugh \cite{Rug05} established 
real analyticity for random repellers. We will see in this note that the situation is completely 
different in the non-autonomous setting. 

\

Bowen's and Ruelle's results are valid for hyperbolic deterministic functions
and hyperbolic functions are so called \emph{stable} functions of the parameter space.
In general, it is not possible to expect nice behavior of the Julia sets and of the dimensions of
these sets  if we perturb an unstable map. Therefore, we first investigate and characterize
stability of non-autonomous maps. 

There are several notions of stability. We consider \emph{holomorphic stability}
that is based on the concept of holomorphic motions and the $\l$--Lemma, which has its origin in
the fundamental paper \cite{MSS83} by Man\'e, Sad and Sullivan.
A parameter $\basept\in \set_0^\N$ is called \emph{ holomorphically stable}
if there exists a family of
holomorphic motions $\{h_{\shift^n(\l)}\}_n$ over some neighborhood $V_{\basept}$ 
such that the following diagram commutes. In here, $\sg(\l_1,\l_2,...)=(\l_2,\l_3,...)$ is the usual shift map.
\beq \label{diag 1}
\xymatrix{
\J_{\basept} \ar[d]_{h_\l} \ar[r]^{f_{\basept_{1}}} & \J_{\shift(\basept)} \ar[d]_{h_{\shift(\l)}}\ar[r]^{f_{\basept_{2}}} &
 \J_{\shift^2(\basept)} \ar[r]^{f_{\basept_{3}}} \ar[d]_{h_{\shift^2(\l)}} & \J_{\shift^3(\basept)}  \ar[d]_{h_{\shift^3(\l)}} \;\; ...\\
\J_{\l} \ar[r]^{f_{\l_{1}}} & \J_{\shift(\l)} \ar[r]^{f_{\l_{2}}} & \J_{\shift^2(\l)} \ar[r]^{f_{\l_{3}}} & \J_{\shift^3(\l)}   \;\;...
}\eeq
Comerford in \cite{Co08} proved stability for certain hyperbolic non-autonomous polynomial maps.
We establish the following characterization of holomorphic stability. It is valid under natural dynamical conditions
(Julia sets are perfect and the maps are topologically exact; see Definition \ref{topo ex}) which
are necessary in order to exclude some pathological examples.
We would like to mention that the usual theory developed by Man\'e, Sad and Sullivan \cite{MSS83}
is based on the stability of repelling periodic points. Such points do not exists at all in the non
autonomous setting.
Another remark is that the parameter space 
$\set_0^\N$ is infinite dimensional.

\bthm\label{holo stab} Suppose that  $\Lambda\subset \set_0^\N$ is equipped with a 
complex Banach manifold structure.
 Let $f_\basept$, $\basept\in \set$, have perfect Julia sets and suppose that
  $f_\l$ is topologically exact for $\l$ in a neighborhood of $\basept$.
Then, the map $f_\basept$ is holomorphically stable if and only if  
 there exist an open neighborhood $V$ of $\basept$
  and three holomorphic functions $\al _i^n:V\to \cbar$, $i=1,2,3$,
such that 
\beq \label{cond 1.2}
\al _i^n (\l )\in \cJ_{\sg^n(\l)}
 \quad \text{ and } \quad \al _i^n (\l )\neq  \al _j^n (\l)\quad \text{for all $\l \in V$ and $i\neq j$.}
\eeq
\beq \label{cond 1.3}
 f_\l ^{n}\big( \cC_{f_\l ^{n}}\big)\cap \{\al _1^n (\l ),\al _2^n (\l ),\al _3^n (\l )\}=\emptyset
  \quad \text{for all $\l\in V$ and
$n\geq 1$.}
\eeq
\beq \label{cond 1.4}
\text{If }\; \al _i^{n+k} (\l )= f_{\sg^n(\la)}^k (\al _j^n (\l )) \; \text{for some $\l\in V$ then this equality holds 
for all $\l\in V$.}
\eeq
\ethm

\brem \label{rem general}
Throughout the whole scope of this paper  we could have chosen
in each fiber $j\geq 0$ the map $f_{\l_j}$ in a different family $\cF_j$ of rational maps.
In particular, Theorem \ref{holo stab} and the whole Section \ref{sec 3} on holomorphic stability does hold
without any restrictions on these families $\cF_j$, $j\geq 0$. Only starting from 
Section \ref{sec 4} we need some further control like, for example, a uniform bound
on the degree of the functions.
We do not insist for such a generalization simply because the notations are already involved enough.
\erem

This characterization is in the spirit 
of the stability of critical orbits
in the deterministic case, i.e. the stability of orbits
$$ c_\l \mapsto f_\l (c_\l) \mapsto ...\mapsto f_\l^n(c_\l)\mapsto ... $$
where $c_\l$ is a critical point of $f_\l$. By Montel's Theorem,
such an orbit is stable if it avoids three values $\al _1^n (\l ),\al _2^n (\l ),\al _3^n (\l )$
depending holomorphically on $\l$ and staying some definite spherical distance apart.
Such a condition appears in Lyubich's paper \cite{Lyu86}
which itself is based on the previous work by Levin \cite{Le81}. It turns out that this is the 
right point of view for generalizing the characterization of stability to the non-autonomous setting.

\

Hyperbolic random and non-autonomous 
polynomials have been studied for example by Comerford \cite{Co06} and
 Sester \cite{Ses99}. Sumi  considered in \cite{Si97} hyperbolic semi-groups.
  The definition of hyperbolicity is based on a uniform expanding
property, and this is the reason why we will call such maps \emph{uniformly hyperbolic}.
We will consider hyperbolic and uniformly hyperbolic non-autonomous maps. 
Later in the course of the paper we will see that they have normal critical orbits and are therefore 
holomorphically stable provided we equip the
parameter space with the $l^\infty$-topology. 
Using
standard properties of quasiconformal mappings we get the following H\"older continuity result
of the dimensions.

\bthm\label{cor holder}
For every uniformly hyperbolic map $f_\basept$ there is a neighborhood $V$ of $\basept$ in 
$l^\infty(\Lambda_0)$  such that the functions 
$$\l\mapsto \HD (\J_\l )\quad \text{and } \quad \l\mapsto \PD (\J_\l )$$
(in fact all fractal dimensions) are H\"older continuous on $V$
with H\"older exponent $\al(\l ) \to 1$ if $\l$ converges to the base point $\basept$.
\ethm

As already mentioned before, in deterministic as well as in random dynamics one has much more, 
namely, real analytic dependence of the dimension
\cite{Ru82, Rug05}.  Surprisingly it turned out that in the non-autonomous setting the H\"older continuity 
obtained in Theorem \ref{cor holder} is best possible. Indeed we show the following.

\bthm\label{thm irregularity}
Consider the quadratic family
$$
\cF = \Big\{ f_\tau (z)= \tau /2(z^2-1)+1\; , \;\; \tau\in \set _0\Big\} \;\; \text{ where }\;\; \set _0 =\{ |\tau |> 40\}
$$
and let $\set$ be the interior of $\set _0^\N\cap l^\infty (\set _0 )$ for the $l^\infty$--topology.
Then $\set= \set ^{uHyp}$ (see Definition \ref{uniform hyp} ) and the functions 
$$\l\mapsto \HD (\J_\l )\quad \text{and } \quad \l\mapsto \PD (\J_\l )$$
are not differentiable at any point $\basept\in \set$ when equipped with the $l^\infty$-topology.
\ethm

In order to prove this result we first produce conformal measures, introduce and study
fiber pressures and establish an appropriate version of Bowen's formula. 
Considering the family $\cF$ in greater detail we also show that generically the different 
fractal dimensions are not identical.

\bthm\label{gap dimension}
Let $\cF$ and $\set$ be like in Theorem \ref{thm irregularity}. Then, there exists an open 
and dense  set $\Omega \subset \set$ such that
$$\HD ( \J_\l ) < \PD (\J_\l ) \quad \text{for every }\;\; \l \in \Omega \,.$$
\ethm


\pagebreak

\section{Non-autonomous dynamics}

Rational functions are holomorphic endomorphisms of the Riemann sphere $\cbar$
and the spherical geometry is the natural setting to work with. Therefore, all
distances, disks and derivatives will be understood with respect to the spherical metric.

We always assume that $\set_0$ is an open and connected
subset of $\C^d$ for some $d\geq 2$
and that $\cF =\big\{f_\tau\,; \; \tau \in \set_0\big\}$ is a holomorphic family of rational functions
which means that $f_\tau$ is a rational function for every $\tau \in \set_0$ and that $(\tau , z)\mapsto f_\tau (z)$
is a holomorphic map from $\set_0\times \cbar $ to $\cbar$.
We are interested in the dynamics of $$ f_{\l _n}\circ ...\circ f_{\l _2}\circ f_{\l _1} \;\; , \quad n\geq 1\,$$
where the $f_{\l_j}\in \cF$ or, equivalently, the $\l_j\in \set_0$  are arbitrarily chosen.

 Let $ \pi :\set_0^\N\to \set_0$ be the canonical projection on the first coordinate
and let $ \shift :\set_0^\N \to \set_0^\N$ be the shift map $\shift (\l_1,\l_2,...)= (\l_2,\l_3,...)$.
To $\l =(\l_1,\l_2,...)\in \set$ we associate a non-autonomous dynamical system
by first identifying $f_\l$ with $f_{\pi (\l )}=f_{\l_1}$ and then by setting 
$$f_\l ^{n} = f_{\shift ^{n-1} (\l )}\circ ...\circ f_{\shift  (\l )}\circ f_{\l }
:=f_{\l _n}\circ ...\circ f_{\l _2}\circ f_{\l _1} \;\; , \quad n\geq1\, .$$
A straightforward generalization of the deterministic case leads to the following definitions.
The \emph{Fatou set} of $(f_\l^n)_n$ is
$$\F (f_\l)=\left\{ z\in\cbar \; ; \;\;   (f_\l^n)_n \text{ is a normal family near } z \right\}$$
and the \emph{Julia set} $\J(f_\l)=\cbar \setminus \F(f_\l)$. 
Most often there will be only one non-autonomous map $f_\l$ 
associated to the parameter $\l$. Then we will use the simpler notations
$\cF_\l$ and  $\cJ_\l$.
For these sets we have the 
invariance property
\beq \label{invariance}
 f_{\l_j}^{-1} (\J_{\shift ^{j+1}(\l)})=\J_{\shift ^{j}(\l)} \text{ and } f_{\l_j}^{-1} (\F_{\shift ^{j+1}(\l)})=\F_{\shift ^{j}(\l)} 
\;, \; j\geq 1 \,. \eeq

\

Here are some basic definitions and observations concerning 
these non-autonomous dynamical systems.

\blem\label{1 or infinite}
The Julia set $\cJ_\l$ of a non-autonomous map $f_\l$ is either infinite or
there exists $N\geq 0$ such that, for every $n\geq N$,
 $\cJ_{\sg ^n (\l)}$ consists in at most two points.
\elem

\bpf
From the invariance property \eqref{invariance} it is clear that
either all the sets $\cJ_{\sg^n(\l)}$, $n\geq 0$, are simultaneously infinite or finite and that
the sequence $n_\l= \#\cJ_{\sg^n(\l)}$ is decreasing hence stabilising when finite.
Suppose that $\#\cJ_\l <\infty$ and let $N$ be the first integer such that 
$$n_\l= (n+1)_\l \quad \text{for every } \;\; n\geq N\, .
$$
Since, by assumption, the functions of $\cF$ are not injective,
it follows that every point of $\cJ_{\sg^N(\l)}$ is a
totally ramified point of  $f_{\sg^N(\l)}$. Therefore we are done since a rational map of degree
at least two has at most two such points.
\epf

As usually, $\cJ_\l$ is called \emph{perfect} if it does not have isolated points.
In the case where $\cJ_\l$ is an infinite set then it is automatically perfect
provided the map satisfies the following mixing property.

\bdfn\label{topo ex}
A map $f_\l$ is topologically exact if, for every open set $U$ that intersects $\cJ_\l$,
there exists $N\geq 1$ such that $f_\l^N(U)\supset \cJ_{\sg^N (\l)}$.
\edfn

As we will see in Example \ref{example}, non-autonomous maps need not be topologically
exact. However, this mixing property is satisfied in most natural settings and is a mild 
natural dynamical condition. B\"uger \cite{buger01} showed that polynomial non-autonomous maps 
with bounded coefficients are topologically mixing.  
This results suggest most likely that $f_\l$ is topologically exact if $\{\l_j\}_j$ is pre-compact in $\set_0$.

\

Non-autonomous maps are very general and 
many of the basic properties valid in the deterministic case are no longer true here.
For example, in the deterministic case a point is in the Julia set if \emph{no} subsequence of the iterates
is normal. Also, deterministic Julia sets are known to be perfect sets. Both these properties are
no longer true in the non-autonomous setting. To illustrate this and some other particularities
we provide here two simple examples.

\bexam\label{example} \rm
Let $f(z)=z^2$ and $h_j(z)=\al_j z$ for some $\al_j>0$, $j\geq 0$.
There are numbers $\l_j>0$ such that for every $ j\geq 1$
\beq \label{ex 1}
h_j \circ f= f_{\l_j}\circ h_{j-1} \text{ where } f_{\l_j}(z)=\l_j z^2 \; .
\eeq
In other words, the deterministic map $f$ is conjugated by the similarities
$(h_j)_j$ to the non-autonomous map $f_\l$. 
The numbers $\al_j$ can be chosen such that
$f_\l ^n(z)= f^n(z)=z^{2^n}$ for even $n$ and  $f_\l ^n(z)= r_nf^n(z)=r_nz^{2^n}$ for odd $n$.
In here the coefficients $r_n$ are chosen to decrease to zero so fast that the sequence
$(f_\l^n)_{n \; odd}$ is normal at every finite point $z\in \C$. Notice that then $(f_\l^n)_{n \; odd}$
is not normal at infinity from which easily follows that 
$$\J_\l = \cS^1 \cup \{\infty \}\,.$$
In particular, this example shows that the conjugation \eqref{ex 1} does not preserve the Julia sets.
Also, the initial system is perfect and topologically exact whereas the new non-autonomous 
map has neither of these properties.
\eexam

\bexam\label{example 2} \rm
Consider $f$ a hyperbolic rational function such that the Fatou set of $f$ has infinitely many
distinct connected components $U_1,U_2,...$. For example, one might take $f(z)=z^2+c$ where $c=-0.123+0.745i$
and where the associated Julia set $\cJ(f)$ is Douady's rabbit. Now, similarly to the first example,
we will modify this deterministic map by conjugating it to a non-autonomous map $f_\l$
where
$$f_{\l_n}= \cM_{n+1}\circ f\circ \cM_n^{-1}\,.$$
This times, $\cM_n=Id$ for even $n$ and, for odd $n$, $\cM_n$ is a M\"obius transformations 
of the Riemann sphere such that $\cM_n(U_n) \supset \cbar \setminus D(0,r_n)$
where $r_n\to 0$.

Notice that $f_{\sg ^{2k}(\l )}^2= f^2$ for every $k\geq 0$. It follows that the deterministic set $\cJ(f)$
is a subset of the non-autonomous set $ \cJ_\l$. On the other hand, it is easy to see that 
$\cF(f)\subset \cF_\l$. Therefore, both systems have the same Julia set $\cJ(f)=\cJ_\l$.

In this example, the conjugation preserves the Julia and Fatou sets.
However, although we started from a hyperbolic hence expanding function $f$, for the non-autonomous
map $f_\l$ we have that 
$$|(f_\l^{2k+1})'|\to 0 \quad \text{ on } \quad \cJ_\l$$
 provided the numbers $r_{2k}\to 0$ sufficiently fast.
\eexam

Further examples with pathological  properties can be found 
e.g. in \cite{Bruck3} and especially in the very interesting papers
\cite{Si10, Si11} by H. Sumi.

Both above examples are obtained in conjugating a deterministic map.
The reason why in both cases the resulting dynamics differ from the original
ones is the  the lack of equicontinuity of the conjugating family of similarities
or M\"obius transformations respectively.
Given this observation it is natural to introduce the following definition.

\bdfn\label{1}
Two non-autonomous maps $f_\l$ and $f_\mu$ are 
 conjugated
if there are homeomorphisms $h_j:\cbar \to \cbar $ such that 
\beq \label{33} h_{j+1}\circ f_{\l_j} =f_{\mu_{j}}\circ h_j\quad \text{holds on $\oc$ for every } j\geq 1\;.
\eeq
If in addition
 the families $\{h_j\}_j$ and $\{h_j^{-1}\}_j$ are equicontinuous
then $f_\l$ and $f_\mu$ are called \emph{bi-equicontinuous conjugated.} 
In the case the homeomorphisms $h_j$ being
(quasi)--conformal  then 
 we say that
the maps are (quasi)--conformally  conjugated or (quasi)--conformally bi-equicontinuous conjugated.
\edfn

The notion of bi--equicontinuous conjugation is consistent with the notion of
 affine conjugations used by Comerford in \cite{Co03}.

Often it is necessary to consider conjugations that do only hold on the Julia sets.
But, in order to do so, it is necessary to first ensure that the conjugating maps do
identify the Julia sets. Clearly, bi-equicontinous conjugations have this property.
As we have seen in Example \ref{example}, conjugations may not. Nevertheless,
in some special cases like in the Example \ref{example 2} Julia sets are preserved.
 Here is a more general statement where this also holds.

\blem[Rescaling Lemma] \label{identif julias}
Suppose that $f_\l$ is a topologically exact non-autonomous map such that all the Julia
sets $\cJ(f_{\sg^n(\l)})$, $n\geq 0$, contain at least three distinct points.
Suppose that $h_n$ are homeomorphisms of $\cbar$ such that $0,1,\infty \in h_n(\cJ(f_{\sg^n(\l)}))$
and such that $(h_n)_n$ conjugates $f_\l$ to the non-autonomous map $g_\l$. Then 
$$ \cJ(g_{\sg^n(\l)}) = h_n(\cJ(f_{\sg^n(\l)})) \quad \text{ for every} \quad n\geq 0\,.$$
\elem

\bpf
It suffices to establish the required identity for $n=0$, i.e. we have to show that
$\cJ(g_\l)= \tilde \cJ_\l$ if $\tilde \cJ_\l= h_0(\cJ(f_\l))$.
Let $\al_1^n,\al_2^n,\al_3^n\in \cJ_{\sg^n(\l)}$ be the points that are mapped by $h_n$
onto $0,1,\infty$ respectively.
If $\tilde z\not\in \tilde \cJ_\l$ then it is easy to see from the conjugations that $\tilde z$ has an open neighborhood
$U$ such that $g_\l^n(U)$ does not contain any of the points $0,1,\infty$. Therefore, Montel's Theorem
yields that $\cbar \setminus \tilde \cJ_\l\subset \cF(g_\l)$ or, equivalently, that $\cJ(g_\l)\subset \tilde \cJ_\l$.

Suppose now that there exists $\tilde z\in \tilde \cJ_\l\cap \cF(g_\l)$. Then there exists an open neighborhood
$U$ of $\tilde z$ such that $(g_\l^n)_n$ is normal on $U$. Let $\ph$ be the limit on $U$ of a convergent 
subsequence of $(g_\l^n)_n$. Shrinking $U$ if necessary, we may assume that
 one of the points $0,1,\infty$ is not in $\ph (U)$.
 Let $\tilde W$ be an open neighborhood of $\tilde z$ such that $\tilde W$ is relatively compact in $U$.
Since $z= h_0^{-1}(\tilde z)\in \cJ (f_\l)$,  the open set $W= h_0^{-1}(\tilde W)$
intersects $\cJ (f_\l)$. By assumption, the map $f_\l$ is topologically exact. Therefore, there is $N>0$ such that
$f_\l^n(W)\supset \cJ_{\sg ^n(\l)}$   for every $n\geq N$. It follows that $g_\l^n (\tilde W ) \supset \{0,1,\infty\}$
for every $n\geq N$. But then we get the contradiction that $\{0,1,\infty\}\subset \ph (U)$.
We showed that $\tilde \cJ_\l\subset \cJ(g_\l) $ and thus both sets coincident.
\epf




\section{Stability and normality of critical orbits}\label{sec 3}

In this section we study holomorphic stability and establish, in particular, Theorem \ref{holo stab}.
We would like to mention that Comerford in \cite{Co08} has a partial result in this direction.
He shows holomorphic stability for certain polynomial non-autonomous systems provided they are
hyperbolic. Our result is an if and only if condition for the stability of a general non-autonomous 
rational map. The condition relies on the dynamics of the critical orbits and,
due to the great generality of non-autonomous systems,
we are lead to consider two different conditions of normal critical orbits. In the Proposition 
\ref{prop 1} and in Theorem \ref{thm 3} we relate them to holomorphic stability
and they yield Theorem \ref{holo stab}.

In the following  we suppose that $\set\subset \set_0^\N$ is a complex Banach manifold.
A canonical choice is to take $\set = \set_0^\N$ and to equip this space with the Tychonov topology.
A more relevant example is to work with the $l^\infty$-topology. Given any function
$\omega: \N\to ]0,\infty[$ then we can take $\Lambda$ the interior
of $\set_0^\N \cap l_\omega^\infty (\C^d)$ in $l_\omega^\infty (\C^d)$ (remember that $\set_0
\subset \C^d$) where the weighted sup-norm
is given by $\|\l\|_{\omega , \infty}:= \sup_j |\omega (j) \l_j|$.
 Denoting this space $\Lambda= l_\omega^\infty (\set _0)$,
then a sequence $\l\in  \set_0^\N$ belongs to $l_\omega^\infty (\set _0)$ if and only if 
$(\omega(1)\l_1,\omega(2)\l_2,..)$ is a bounded sequence such that $\inf _j \omega(j)dist(\l_j , \partial \set_0 )>0$.

Starting from Section 4 we most often deal with uniform hyperbolic maps (see Definition \ref{uniform hyp}).
Then the natural associated parameter space is $\set = l^\infty (\set _0)$, i.e. the space
$l_\omega^\infty (\set _0)$ with weight function $\omega \equiv 1$.

\subsection{Holomorphic motions}

Since this section relies on quasiconformal mappings
and holomorphic motions, we start by summarizing some facts from this theory.
Let $\basept\in \Lambda$ be a base point.

\bdfn\label{holo motion}
A \emph{holomorphic motion} of a set $E\subset \cbar$ over $\Lambda$ is a mapping
$h:\Lambda \times E \to \cbar$ having the following three properties.
\begin{itemize}
\item $h_{\basept}=id_E$,
\item for every $\l\in \Lambda$, the map $z\mapsto h_\l (z)$ is injective on $E$ and
\item for every $z\in E$, $\l\mapsto h_\l (z)$ is a holomorphic map on $\Lambda$ .
\end{itemize}
\edfn

As already mentioned in the introduction, Man\'e, Sad and Sullivan \cite{MSS83} initially established 
a $\l$--Lemma stating that 
any holomorphic motion of a set $E\subset \cbar$ over the unit disk of $\C$ can be extended 
to a holomorphic motion of the closure of $E$. Since then, this $\l$--Lemma
has been extensively studied and generalized. Most notably, 
Slodkowski \cite{Slo95} showed that every holomorphic motion over the unit disk is the restriction of a holomorphic motion of the whole sphere.
Hubbard \cite{Hubbard} discovered that this is false for holomorphic motions over higher-dimensional
parameter spaces and \cite{YaMi07} contains a simpler example.
Nevertheless, we dispose in the following $\l$--Lemma due to Mitra \cite{Mi00}
and Yiang-Mitra \cite{YaMi07}.

\bthm[$\l$--Lemma] \label{l-Lemma}
A holomorphic motion $h$ of a set $E\subset \cbar$ over a 
simply connected complex
Banach manifold $V$
with basepoint $\basept\in V$ extents to a holomorphic motion $H$ of $\overline E$
over $V$ such that 
\ben
\item for every $\l\in V$, the map $H_\l$ is a global quasiconformal map of $\cbar$
with dilatation bounded by $\exp (2\rho _V (\basept , \l ))$ where $\rho_V$ is the Kobayashi
pseudometric on $V$.
\item the map $(\l, z)\mapsto H_\l (z)$ is continuous.
\een
\ethm

\subsection{Holomorphic stability and normal critical orbits}
Here is the precise definition of the stability we use. Notice that, in this definition, the conjugating
maps $h_{\shift^n(\l)}$ are not necessarily bi-equicontinuous. We therefore have to include here that the conjugating maps 
identify the Julia sets.

\bdfn\label{h stability}
A map $f_\basept$, $\basept\in \set$, is  holomorphically stable if there is an open 
neighborhood $V\subset \set$ of $\eta$ and a 
family of holomorphic motions 
$\{h_{\shift^n(\l)}\}_n$ of $\{\J_{\shift^n(\basept)}\}_n$ over $V$ such that,
for every $\l\in V$, $h_{\shift^n(\l)} ( \J_{\shift^n(\basept)}) = \J_{\shift^n(\l )}$
and
$$h_{\shift^{n+1}(\l)} \circ f_{\shift^n(\basept)} =f_{\shift^{n+1}(\l)}\circ h_{\shift^n(\l)}
\quad \text{on} \quad \J_{\shift^n(\basept)} \;\; \text{for every} \;\; n\geq 0\,.$$
The set of holomorphic stable parameters is denoted by
$\set^{stable}$.
\edfn

In the theory by Man\'e, Sad and Sullivan \cite{MSS83} and, independently, Lyubich \cite{Lyu86},
 showing in particular density
of stable parameters in any deterministic holomorphic family
of rational functions, appear several equivalent characterizations
of stability. Most of this theory relies
heavily on the stability of repelling cycles which, in the present non-autonomous setting, do not exist at all.
There is one criterion of stability in  \cite{Lyu86} which turns out to be appropriate for generalization
to the present setting. This criterion exploits
the dynamics of the critical orbits
$c_\l\mapsto f_\l(c_\l)\mapsto ...\mapsto f^n_\l(c_\l) \mapsto ...$ under perturbation of $\l$. Indeed,
stability coincides with the normality of these orbits and, as already mentioned in the introduction,
Montel's Theorem implies that
such an orbit is stable if it avoids three values $\al _1^n (\l ),\al _2^n (\l ),\al _3^n (\l )$
depending holomorphically on $\l$ and staying some definite distance apart.
It is therefore natural to make the following definition.

\bdfn \label{NCO}
A map $f_{\basept}$ has
\emph{normal critical orbits} on $V$, an open neighborhood of $\basept$, if
there exist $\kappa >0$ and,
 for each $n\geq 0$, three holomorphic functions $\al _i^n:V\to \cbar$, $i=1,2,3$,
such that 
\beq \label{cond 3.2}
dist_S( \al _i^n (\l ) , \al _j^n (\l))\geq \kappa \quad \text{for all $\l \in V$ and $i\neq j$.}
\eeq
\beq \label{cond 3.3}
 f_\l ^{n}\big( \cC_{f_\l ^{n}}\big)\cap \{\al _1^n (\l ),\al _2^n (\l ),\al _3^n (\l )\}=\emptyset
  \quad \text{for all $\l\in V$ and
$n\geq 1$.}
\eeq
\beq \label{cond 3.4}
\text{If }\; \al _i^{n+k} (\l )= f_{\sg^n(\la)}^k (\al _j^n (\l )) \; \text{for some $\l\in V$ then this equality holds 
for all $\l\in V$.}
\eeq
 \edfn

Notice that \eqref{cond 3.3} is precisely \eqref{cond 1.3} and the compatibility condition \eqref{cond 3.4}
is also exactly the condition \eqref{cond 1.3} of Theorem \ref{holo stab}. 
Only the first condition \eqref{cond 3.2} differs from the corresponding one in Theorem \ref{holo stab}. 
It is a normalized version of condition \eqref{cond 1.2} 
in which we allow the functions $\al_j^n$ to have values not only in the corresponding Julia set but in the 
whole Riemann sphere.
If, in this definition, the condition \eqref{cond 3.2} is replaced by \eqref{cond 1.2}, then we 
will say that $f_{\basept}$ has
\emph{normal critical orbits in the sense of Theorem \ref{holo stab}} on $V$.

\bprop \label{prop 1}
Suppose that $\basept\in \set^{stable}$ is a holomorphic stable parameter
and that $\cJ_\basept$ is a perfect set. Then $f_\basept$ has normal critical orbits 
in the sense of Theorem \ref{holo stab}.
\eprop

\bpf
Consider first the map $f_\basept$ and let us define the points $\al_j^n(\basept )$
by induction. Since $\cJ_\basept$ is perfect, there exist
three distinct points $\al_1^0(\basept ), \al_2^0(\basept ), \al_3^0(\basept )\in \cJ_\basept$. 
Suppose that all the points $\al_j^k(\basept )$ are defined for $0\leq k <n$.
The set $\cJ_{\sg^n(\basept )}$ is also perfect and so there are distinct points
$$\al_1^n(\basept ), \al_2^n(\basept ), \al_3^n(\basept )\in \cJ_{\sg^n(\basept )}
\setminus \left[ f_\basept ^{n}\big( \cC_{f_\basept ^{n}}\big)  \cup 
 \bigcup_{k=0}^{n-1}  f_{\sg ^k(\basept)}^{n-k} (\al _j^k (\basept)) \right]\,.$$
By assumption there are holomorphic motions $\{h_{\shift^n(\l)}\}_n$ such that 
Definition \ref{h stability} is satisfied. It suffices now to set
$$ \al_j^n (\l):= h_{\shift^n(\l)} (\al_j^n (\basept)) \quad \text
{for every} \;\; \l\in V \text{ and all } n,j\,.$$

\epf

The following main result of this section goes in the opposite direction.
Notice that here we do not need any additional assumption. So, in particular,
 no topological exactness is needed.

\bthm\label{thm 3}
Suppose that $f_\basept$ has normal critical orbits.
Then $f_\basept$ is holomorphically stable, i.e. $\basept \in \set^{stable} $.
Moreover, the corresponding family of holomorphic motions  is bi--equicontinuous;
it gives rise to a bi--equicontinuous conjugation.
\ethm

Before giving a proof of it, let us first explain how Theorem \ref{holo stab}
results.

\bpf[Proof of Theorem \ref{holo stab}]
Given Proposition \ref{prop 1} we only have to show that normality of critical orbits
in the sense of Theorem \ref{holo stab} implies holomorphic stability.
Let $f_\basept$ be a map such that there exist functions $\al_1^n , \al_2^n ,  \al_3^n $
defined and holomorphic on some neighborhood $V$ of $\basept$
such that the conditions \eqref{cond 1.2}, \eqref{cond 1.3} and \eqref{cond 1.4}
are satisfied.
Let $\cM_{\sg^n(\l )}$ be a M\"obius transformation
sending the points $\al_j^n (\l)$, $ j=1,2,3$, to $0,1,\infty$ and consider $ \tilde f_{\sg^n(\l )}$
defined by 
\beq \label{eq 11}\tilde f_{\sg^n(\l )} \circ \cM_{\sg^n(\l )} = \cM_{\sg^{n+1}(\l )} \circ f_{\sg^n(\l )}
\quad \text{for every } \; \l\in V \;\; \text{and } n\geq 0\,.\eeq
By assumption, $f_\l$ is topologically exact near $\basept$, say on $V$.
Therefore, Lemma \ref{identif julias} applies and yields that
$$ \cJ(\tilde f_{\sg^n(\l )} ) =\cM_{\sg^{n}(\l )}\Big(\cJ (f_{\sg^n(\l )})\Big)\quad 
\text{ for all }\;\; \l \,,  n \;.$$
Since the functions $\l \mapsto \al_j^n(\l )$ are holomorphic on $V$, it suffices to 
establish holomorphic stability of $\tilde f_\basept$. 
This new function $\tilde f_{\basept }$ has normal critical orbits
(with functions $\tilde \al_j^n $ constant $0, 1$ or $\infty$) and so we would like to conclude
by applying Theorem \ref{thm 3}. 
However, on every fiber the map
$ \tilde f_{\sg^j(\l )}$, $j\geq 0$, belongs to a different holomorphic family
$\cF_j=\{ \tilde f_{\sg^j(\l )}\, ; \; \l \in V\}$. But, as already mentioned in Remark \ref{rem general},
the whole paper
 and especially Theorem \ref{thm 3} does hold in this generality
with the same proof. Therefore $\tilde f_{\basept }$ is holomorphically stable.
\epf


The remainder of this section is devoted to the proof of Theorem \ref{thm 3}.
In order to do so, suppose from now on that $f_\basept$ has normal critical 
orbits: there are $V$, an open neighborhood of $\basept$, and holomorphic 
functions $\al_j^n$ such that the conditions of Definition \ref{NCO} are satisfied.
Consider  the sets
$$E_{\sg ^j(\l),n}=
f_{\sg^j(\l )} ^{-(n-j)} \big( \{\al _1^n (\l ),\al _2^n (\l ),\al _3^n (\l )\}\big) \quad , \quad j\leq n$$
and
\beq \label{set e}
\cE_{\sg ^j(\l)}=\bigcup_{n\geq j} E_{\sg ^j(\l),n} \quad , \quad \l\in V
\text{ and } j\geq 0 \;.
\eeq

\bprop\label{pf 3}
For every $j\geq 0$, there are holomorphic motions $h_{\sg^j(\l)}: \cE_{\sg ^j(\basept)}
\to \cE_{\sg ^j(\l)}$ over $V$ such that
\beq \label{pf 3.1}
h_{\sg^j(\l)}(\al_i^j(\basept))= \al_i^j(\l)\quad \text{for all $\l\in V$ and $i\in \{1,2,3\}$ and}
\eeq
\beq \label{pf 3.2}
h_{\sg^{j+1}(\l)}\circ f_{\sg ^j (\basept )}= f_{\sg ^j (\l  )}\circ 
h_{\sg_j (\l )} \quad\text{ on $\;\;\cE_{\sg ^j(\basept)}$, $\;\l\in V$.}
\eeq
\eprop

\bpf
We explain how to obtain the motions in the case $j=0$. The general case is proven
exactly the same way.

Let $z_\basept\in \cE_\basept$ and let $n\geq 0$ be minimal such that $z_\basept\in E_{\basept, n}$.
A point  $z_{\basept}\in E_{\basept,n }$ if $f_{\basept}^n (z_{\basept})= \al _i^n(\basept )$ for some 
$i\in \{1,2,3\}$. Hence, we have to consider  the equation 
\beq \label{impl ft} f_{\l}^n (z)= \al _i^n(\l )\,.\eeq
We want to apply the implicit function theorem to this equation and get $z$ as a function
of $\l$. This is possible as long as $(f_{\l}^n)' (z)\neq 0$. If $(f_{\l}^n)' (z)= 0$,
then the point $ \al _i^n(\l )$ is a critical value of $f_{\l}^n$. However, the assumption
\eqref{cond 3.3} implies that this is not the case for $\l\in V$. Therefore there
is a uniquely defined holomorphic function $\l\mapsto z_\l$, $\l\in V$, starting at the 
given point $z_{\basept}$, if $\l=\basept$, and such that $(\l, z_{\l})$ is solution of \eqref{impl ft}.
Therefore, we can define $$h_\l (z_\basept)= z_\l\;\; ,\quad \l\in V\,.$$

If ever $z_{\l}\in E_{\l,k } \cap E_{\l ,n }$ for some $\l\in V$ and $1\leq k\leq n$, then 
there are $i,j\in\{1,2,3\}$ such that $\al _i^n(\l )
=f_{\sg^k(\l )}^{n-k} (\al _j^k(\l ))$.
But then the compatibility condition \eqref{cond 3.4} implies that 
the last equation holds for all $\l\in V$ and that it does not matter for
the definition of the function $\l \mapsto z_\l$ if we start with
$\al _i^n(\eta )$ or with $\al _j^k(\eta )$.
 
  The normalization \eqref{pf 3.1} and the conjugating relation \eqref{pf 3.2} are clearly
 satisfied simply by the way we constructed the holomorphic motions. Hence, the proof is complete.
\epf

We are now able to conclude the proof of Theorem \ref{thm 3} since
we can now apply Mitra's version of the $\l$--Lemma.
 Indeed, Theorem \ref{l-Lemma} asserts that the motions $h_{\sg^j (\l )}$ extend 
 to holomorphic motions of the closure $\overline \cK_{\sg^j (\l )}$. We continue
 to denote these extended motions by $h_{\sg^j (\l )}$. These maps
 $h_{\sg^j (\l )}$ are global quasiconformal homeomorphisms with 
 dilatation bounded by  $\exp (2\rho _V (\basept , \l ))$. Therefore, for every
 fixed $\l\in V$ the family  $(h_{\sg^j (\l )})_j$ is uniformly quasiconformal
 and normalized by \eqref{pf 3.1}. Since the points $\al_i^j (\l)$, $i=1,2,3$,
are at definite spherical distance (see Condition \eqref{cond 3.2}),
it results from standard properties of families of uniformly quasiconformal
mappings that the conjugation by  $(h_{\sg^j (\l )})_j$ is bi-equicontinuous.

Up to now we showed that Theorem \ref{thm 3} holds but with the julia sets
$\cJ_{\sg^j (\l )}$ replaced by the sets $\overline \cK_{\sg^j (\l )}$.
However it is not hard to see that $\cJ_{\sg^j (\l )}\subset \overline \cK_{\sg^j (\l )}$.
Indeed, for every open set $U\subset \cbar \setminus \overline \cK_{\sg^j (\l )}$ 
we have that
$$f_{\sg^j (\l )}^n (U) \cap \left\{ \al _1^{j+n}(\l), \al _2^{j+n}(\l),\al _3^{j+n}(\l)\right\}
=\emptyset \quad \text{for every } \, n\geq 0\,.$$
Hence, Montel's Theorem along with Condition \eqref{cond 3.2} imply that $U\subset \cF_{\sg^j (\l )}$.
Consequently, $\cJ_{\sg^j (\l )}\subset \overline \cK_{\sg^j (\l )}$
for every $j\geq 0$. The proof of Theorem \ref{thm 3} is complete.

\

From this study of holomorphic stability we get first informations concerning 
our initial problem, namely the behavior of the variation of the Julia sets
and of their dimensions. 

\bcor \label{cor 11}
Suppose that $\Lambda\subset \set_0^\N$ is a complex Banach manifold
and let $\basept \in \Lambda^{stable}$. Then, in some neighborhood of $\basept$ in $\Lambda$,
the function $\l \mapsto \cJ_\l$
is continuous and $\l \mapsto HD(\cJ_{\sg^j (\l )})$ as well as $\l\mapsto BD(\cJ_{\sg^j (\l )})$
are H\"older continuous with H\"older constants depending on $\l$ only.
\ecor

\bpf
The assertion on the H\"older continuity directly results from known properties of quasiconformal
mappings along with the fact that the distortions of the quasiconformal mappings
$h_{\sg^j (\l )}$ do only depend on $\l$ and not on $j\geq 0$.
Concerning the continuity of the Julia sets, this is a consequence of the continuity of 
the function
$(\l , z) \mapsto  h_{\sg^j (\l )}(z)$  (see property (2) of Theorem \ref{l-Lemma}).
\epf

\section{Hyperbolic non-autonomous systems} \label{sec 4}

 In deterministic dynamics a hyperbolic function is stable. 
 But if we perturb a deterministic hyperbolic function to a non-autonomous map
then the stability depends on the topology we use on the parameter space.
As an illustration we first consider the simple Tychonov convergence and 
explain that, for this topology,  every map
is unstable (see Proposition \ref{stable tychonov}).

Then we investigate non-autonomous hyperbolic and uniform hyperbolic functions and will see that the
later are stable
provided the parameter space is $\set=l^\infty (\set _0)$.
In oder to prove their stability  it suffices to use Theorem \ref{thm 3}. Indeed, the
 normal critical orbits condition is best appropriated since it is easy to check for hyperbolic
 maps.

\subsection{Stability and Tychonov topology}
Up to here, the parameter space $\Lambda$ was equipped with any arbitrary 
complex manifold structure. Let us inspect a particular case.

\bprop\label{stable tychonov}
Suppose that $\cF$ contains at least two deterministic hyperbolic maps having Julia sets
with different Hausdorff dimension. Suppose further that
$\set =\set_0^\N$ and that $\set$ is equipped with
the 
Tychonov structure induced by the simple convergence.
Then $$\Lambda^{stable} = \emptyset\;.$$
\eprop

\bpf
Let $\basept \in \Lambda$ and set $\d = HD(\cJ_\basept )$.
By hypothesis there exists $f_{\l_0}\in \cF$ a deterministic hyperbolic map
with  $\d'=HD(\cJ(f_{\l_0})\neq \d$.
Consider then 
$$\l^{(n)}=(\basept _1, \basept_2, ..., \basept_n, \l_0,\l_0,\l_0,...)\,.$$
On the one hand we have that $\l^{(n)} \to \basept$ point wise. On the other hand
we have $HD(\cJ_{\l^{(n)}})=\d'$ for every $n\geq 1$ and hence $HD(\cJ_{\l^{(n)}})\not\to HD(\cJ_\basept )$
as $n\to \infty$. But then it follows from Corollary \ref{cor 11} that $\basept$ cannot be a stable
parameter.
\epf

\subsection{Hyperbolicity}

Hyperbolic random systems have been studied in various 
papers (see e.g. \cite{Co06,Ses99} and also \cite{Si97} where hyperbolic 
semi-groups are considered). 
In these papers, normalized most often polynomial families are considered
and the definitions of hyperbolicity rely on uniform conditions.
We therefore call such functions \emph{uniformly} hyperbolic.

\bdfn \label{uniform hyp}
A  map $f_\l$ is \emph{uniformly hyperbolic} if the family $\left\{f_{\l_j}\,;\; j\geq 1\right\} $
is equicontinuous (which, for example, is the case if
$\{\l_j\,,\; j\geq 1\}$ is relatively compact in $\set_0$ or, equivalently, if $\l \in l^\infty (\set_0)$ ) and if
there exist $c>0$ and $\g>1$ such that for every $j\geq 0$ we have
\beq  \label{expanding} |(f_{\shift^j(\l )}^n)' (z)| \geq c\g ^n \quad \text{for all} \;\; z\in \J_{\shift^j(\l )} \;\; \text{and }\; n\geq 1\;.\eeq
The set of parameters of uniformly hyperbolic random maps is denoted by $\set^{u Hyp}$.
\edfn

For general families of non-autonomous maps this definition is not entirely satisfactory.
For instance, in the Example \ref{example 2} we have conjugated a deterministic hyperbolic
function by M\"obius maps. The resulting non-autonomous map does not satisfy the 
requirements of Definition \ref{uniform hyp}
although it shares many properties of maps that should be called hyperbolic.
It is uniformly expanding "up to a conformal change of coordinates".
Moreover, it is topologically exact which, as we will see (Lemma \ref{mix}), is a property that
uniform hyperbolic maps always have.

A natural candidate for the class of hyperbolic maps is to take all the maps that are M\"obius conjugate
to uniform hyperbolic maps. However, one has to be careful since the map given in Example \ref{example},
obtained by conjugation by similarities of a deterministic hyperbolic
function,
should really not be called hyperbolic.
Given these examples and Lemma \ref{identif julias} which ensures that the Julia sets are identified provided the dynamics are topologically exact, it is natural to introduce the following definition.

\bdfn\label{hyp maps}
A non-autonomous map $f_\l$ is \emph{hyperbolic}
if it is topologically exact, if $\# \cJ_{\sg^j(\l )} \geq 2$ for all $j\geq 0$ and if there are M\"obius transformations 
conjugating $f_\l$ to a uniformly hyperbolic map.
\edfn


We now consider uniform hyperbolicity greater in detail. 
Let
$\cV_\d (E)=\{z\; ; \; dist (z,E)<\d \}$ be the $\d$-neigborhood of the set $E$.

\blem \label{improving expansion}
The map $f_\l$ is uniformly hyperbolic if and only if the family $\left\{f_{\l_j}\,;\; j\geq 1\right\} $
is equicontinuous and there exist $\d >0$, $N\geq 1$ and  $\tau >1$ such that 
\beq  \label{d-expanding} 
|(f_{\shift^j(\l )}^N)' (z)| \geq \tau >1 \quad \text{for all} \;\; z\in \cV_\d(\cJ_{\shift^j(\l )}) \;\; \text{and }\; j\geq 0\;.\eeq
In particular, if $f_\l$ is uniformly hyperbolic then there exist $\d >0$ such that 
for all $n\geq 1$, $j\geq 0$ and $z\in \cJ_{\shift^{n+j}(\l )}$
all holomorphic inverse branches of $f_{\shift^j(\l )}^n$ are well defined on $D(z,\d)$ have uniform distortion
and are uniformly contracting.
\elem

\bpf
Suppose that $f_\l$ is uniformly hyperbolic and fix $N\geq 1$
such that $c\g^N>1$. Suppose that \eqref{d-expanding} does not hold.
More precisely, suppose that for any $\d>0$ and any $1<\tau < c\g^N$
there exist $w=w_{\d , \tau }\in \cV_{\d}\left(\cJ_{\shift^j(\l )}\right)$
for some $j=j_{\d , \tau }\geq 0$ such that 
$$| (f^N_{\shift^j(\l )})'(w) | \leq \tau\,.$$
Let $z_{\d, \tau }\in \cJ_{\shift^j(\l )}$ such that $|z_{\d, \tau }-w_{\d, \tau }|<\d$.
Due to the equicontinuity of the family $\left\{ f_{\sg^j{(\l)}}^N\; , \; j\geq 0\right\}$
we can choose sequences $\d_n\to 0$, $\tau _n \to 1$ such that the 
corresponding functions $ f^N_{\shift^{j(n)}(\l )}\to \ph$ and  points $w_{\d_n, \tau _n} \to \xi$,
$z_{\d_n , \tau_n }\to \xi$
converge as $n\to \infty$. But then it is easy to see that $|\ph '(\xi )|\leq 1$ and, in the same
time, $|\ph '(\xi )|\geq c\g^N>1$. This contradiction shows that uniform hyperbolicity
implies \eqref{d-expanding}. The other assertion results now from standard arguments.
\epf


In the case of  deterministic iteration of rational functions there are several equivalent conditions
for hyperbolicity. One of them is the \emph{expanding} condition, another condition demands that
critical orbits are captured by attracting domains. Here is a version in the non-autonomous case
which in fact is an adaption of \cite{Ses99}.

\begin{prop} \label{2}
A  map $f_\l$ is uniformly hyperbolic if and only if
there exist $m_0 >0$ and open sets $U_j$ such that, 
for every $j\geq 0$,
\ben
\item $\overline{f_{\shift^j(\l )} (U_j ) }\subset U_{j+1}$ and  $dist_S (f_{\shift^j(\l )} (U_j ) , \partial U_{j+1})\geq m_0$,
\item $D(z, m_0 ) \cap U_{j}=\emptyset$, for every $z\in \cJ _{\shift^j(\l )}$, and
\item the critical points of $f_{\shift^j(\l )}$ are contained in $U_j$.
\een
\end{prop}

\bpf
Since most of the proof is standard we only give a brief outline of it.
Especially, finding the sets $U_j$ knowing that $f_\l$ is uniformly hyperbolic
is a straightforward adaption of Sester's arguments \cite[pp. 414-415]{Ses99}
which themselves are based on the deterministic case. The main idea is 
to build a metric in which all the functions $f_{\shift^j(\l )}$ have a derivative greater than
some constant $\g>1$ on $\cV_\d(\cJ_{\shift^j(\l )})$ for some $\d>0$.

The proof of the opposite implication
 is based on hyperbolic geometry. 
 Suppose the sets $U_j$ are given, set $V_{j+1}=f_{\sg^j(\l)}(U_j)$
 and $\tilde{U}_j=  f_{\sg^j(\l)}^{-1} (V_{j+1})$. Then $f_{\sg^j(\l)}:\tilde{U}_j \to V_{j+1}$ is a proper map
 and, the critical orbits being captured by the domains $U_j$ (see (3)),
  $f_{\sg^j(\l)}: \omega_j\to \Omega_{j+1}$ is a covering map where $\omega_j , \Omega _{j+1}$
 is the complement of the closure of $\tilde{U}_j , V_{j+1}$ respectively.
 Therefore this map is a local hyperbolic isometry with respect to the hyperbolic distances 
 of these domains. Property (1) implies that there is $0<c<1$ such that the inclusion map 
 $i:\omega_{j+1}\to \Omega_{j+1}$ is a hyperbolic $c$--contraction for all $j\geq 0$.
 Combining these properties it follows that $f_{\sg^j(\l)}$ is a $1/c$--expansion
 on $\J_{\sg^j(\l)}\subset \omega_j\cap f^{-1}(\omega_{j+1})$ with respect to the 
 hyperbolic distances of $\omega_j$ and $\omega_{j+1}$. 
 Finally, it results from property (2) that it is possible to compare
 the hyperbolic and spherical distance for points in $\cJ_{\sg^j(\l )}\subset \omega_j$, $j\geq 0$, and to conclude.
 \epf

The topological characterization of Propositon \ref {2} and espacially the 
uniform control due to the constant $m_0$ implies the following.

\bcor \label{4}
 Uniform hyperbolicity is an open condition for the $l^\infty$-topology on $\set$
(but not for the Tychonov topology).  
Moreover, if $\basept \in \set^{hyp}$ then there is an open neighborhood
$V \subset \set^{hyp}$ of $\basept$ such that the open sets $U_j$ 
and the number $\d=\d(\l )>0$ given by Lemma \ref{improving expansion}
can be chosen 
to be the same for all the maps $f_\l$, $\l\in V$.
\ecor

This result immediatley implies the following continuity property of non-autonomous Julia sets
which, in various versions, is well known to the specialists
(see for example \cite{Bruck2, Ses99, Co06}.

\bprop\label{conti tycho}
Every $\basept\in \set^{uHyp}$ has an open  neighborhood $V\subset l^\infty(\set_0 )$
such that 
 the map $$ \l  \longmapsto    \cJ_\l $$ from $ ( V , \text{ Tychonov topology})$
into $(\cK (\cbar ), \text{ Hausdorff topology})$
is continuous.
\eprop

\bpf
Let $\basept\in \set^{uHyp}$ and let the open neighborhood $V$ of $\eta$ be relatively compact in $\set$ 
with respect to the $l^\infty$-topology
and chosen according to Corollary \ref{4}, i.e. 
 there are open sets $U_j$ such that 
every map $f_\l$, $\l\in V$, satisfies the conditions (1), (2) and (3) of 
Proposition \ref{2} with these sets $U_j$. Denote
$${\tilde U}_j =\{z\in U_j\;;\;\; dist_S(z, \partial U_j)>m_0/2\}\,.$$
Shrinking the neighborhood $V$ if necessary and replacing $m_0$ by a smaller constant
 we may assume that
the open sets ${\tilde U}_j $ satisfy also the conditions (1), (2) and (3) of 
Proposition \ref{2} for every $\l\in V$. Moreover, 
all inverse branches exist and are uniformly contracting on the complement of $\tilde U_j$, $j\geq 1$.

Define
$$\cA_\l ^n =\{ z\in \cbar \;;\;\; f_\l ^n (z) \not\in U_n  \}
\quad \text{and}\quad \tilde\cA_\l ^n =\{ z\in \cbar \;;\;\; f_\l ^n (z) \not\in \tilde U_n  \}\,.$$
Clearly  $\J_\l \subset\bigcap _n \cA_\l^n\subset\bigcap _n \tilde \cA_\l^n$. On the other hand, since
all inverse branches exists and are uniformly contracting on the complement of $\tilde U_j$, $j\geq 1$,
we have first of all that 
$\J_\l =\bigcap _n \cA_\l^n=\bigcap _n \tilde\cA_\l^n$ and, secondly,
that for every $\ep >0$ there exist $n=n_\ep\geq 1$ such that $\cA_\l^n\subset \tilde \cA_\l^n\subset \cV_\ep (\cJ_\l)$
for every $\l\in V$.

Fix $\ep>0$ and let $n=n_\ep$. Notice that the sets $\cA_\l^n$ and $\tilde \cA_\l^n$ do only depend on 
the $n$ functions $f_{\l_1},...,f_{\l_n}$. A standard compactness argument shows now that there exists 
$\d=\d(\ep)>0$ such that
$$\cA_\l^n\subset \tilde\cA_{\l'}^n \quad \text{for every} \quad \l,\l'\in V \;\; \text{ such that }
\sup_{i=1,...,n}|\l_i-\l'_i| <\d \,.$$
Therefore, for every $\l,\l'\in V$ such that $\sup_{i=1,...,n}|\l_i-\l'_i| <\d$ we have that
$$\cJ_\l \subset \cA_\l^n\subset \tilde\cA_{\l'}^n \subset \cV_\ep (\cJ_{\l'} )$$
This proves the proposition.
\epf

We conclude the discussion on uniform hyperbolicity with the following 
uniform mixing property.

\blem\label{mix}
Let $\l\in \set^{uHyp}$  and let $\d=\d (\l)$.
Then, for every $r_1>0$ and $0<r_2\leq\d$, there exist $N=N(r_1,r_2 )$ such that
for all $j\geq 0$, $z_1\in \J_{\shift^j (\l )}$ and $z_2\in \J_{\shift^{j+N} (\l )}$ we have that
$$f_{\shift^j (\l )}^N \big(D(z_1,r_1)\big) \supset D(z_2,r_2)\;.$$
In particular, $f_\l$ is (uniformly) topologically exact: for every $r_1>0$ there exist $N=N(r_1)$
such that for $j\geq 0$ and $z_1\in \J_{\shift^j (\l )}$ we have that
 $f_{\shift^j (\l )}^N \big(D(z_1,r_1)\big) \supset \cJ_{\sg^{j+N}(\l )}$.
\elem

\bpf
Suppose to the contrary that there exist $r_1>0$ and $0<r_2\leq\d$
and, for every $N$, $j_N\geq0$, $z_{1,N}\in \J_{\shift^j (\l )}$ and $z_{2,N}\in \J_{\shift^{j_N+N} (\l )}$
such that
\beq \label{eq mix 1}
D(z_{2,N}, r_2)\setminus f^N_{\shift^{j_N}(\l)}\left(D(z_{1,N}, r_1)\right) \neq \emptyset \,.
\eeq
Consider then $\ph _N(z)= f^N_{\shift^{j_N}(\l)} (r_1z+z_{1,N})$, $z\in \D$.
Since $f_\l$ is expanding on the Julia set the family $(\ph_N)_N$ is not normal at the origin.
Therefore there are infinitely many $N$ such that
\beq \label{eq mix 2}
\ph _N(\D(0,1/2))\cap D(z_{2,N}, r_2) \neq \emptyset \,.
\eeq
Since $r_2\leq \d$, all inverse branches of $f^N_{\shift^{j_N}(\l)}$ are well defined and have bounded
distortion on $D(z_{2,N}, r_2)$. It suffices then to choose $N$ big enough and to deduce from expanding
along with \eqref{eq mix 2} that
$$f^{-N}_{\shift^{j_N}(\l),*}\left(D(z_{2,N}, r_2)\right) \subset D(z_{1,N}, r_1)$$
where $f^{-N}_{\shift^{j_N}(\l),*}$ is some well chosen inverse branch. This contradicts \eqref{eq mix 1}.
\epf

\subsection{Hyperbolicity and stability}
The definition of hyperbolic map is based on uniform controls, e.g. 
the iterated maps $f_{\sg^j (\l)}^n$ are expanding uniformly in $j$. 
With respect to this and in order to deal with perturbations of hyperbolic
functions it is natural to equip the parameter space  $\Lambda$
with the sup-norm, i.e. to work with the space $\set=l^\infty (\set_0)$.
 Throughout the rest of this paper
we suppose that $\set$ is this particular Banach manifold.

As already mentioned, in order to establish stability of uniformly hyperbolic maps,
the condition of normal critical orbits 
as defined in Definition \ref{NCO} is perfectly adapted  since
easy to verify for such functions.

\bprop \label{hyp sing}
If $f_{\basept}$ is a uniform hyperbolic map, then $f_{\basept}$ has normal singular orbits 
on some open neighborhood $V\subset \set$ of $\basept$.
\eprop

\bpf
By Corollary \ref{4}, there is an open neighborhood $V \subset \set$ 
such that the open sets $U_n$ in Proposition \ref{2} can be chosen independently on $\l\in V$.
Since we know that 
$$dist_S(f_{\l_n}( U_n), \partial U_{n+1} )\geq m_0$$
we can find  three points $a_i^0\in U_0$ and, if if $n>0$,
 $$a_i^n\in  U_n \setminus\bigcup_{\l\in V} f_{\l_{n-1}} (\overline U_{n-1})$$
such that $dist_S(a_i^n , a_j^n)\geq c_0$ for some 
$c_0>0$ and for all $n\geq 0$ and $i\neq j$.
Since $\cC_{f_{\l_j}}\subset U_j$, $j\geq 1$, we have the inclusion
$f_\l ^n (\cC_{f_\l ^n})\subset f_{\l _n}(U_{n-1}) \subset U_n$.
The constant functions $\l\mapsto \al_i^n(\l ) =z_i^n$, $\l \in V$,
therefore satisfy the conditions (1) and (2) of Definition \ref{NCO}
and appropriate perturbations of these constant functions if necessary
yield that Condition (3) of this definition is also satisfied.
Therefore, $f_\l$ has normal critical orbits on $V$.
\epf

The following statement follows now from Theorem \ref{thm 3}.

\bcor
$\set^{uHyp}\subset \set^{stable}$ when equipped with the $l^\infty$-topology.
\ecor

\section{Conformal measures, pressure and dimensions} \label{sec 5}
In this section we consider a single non-autonomous uniformly hyperbolic map $f_\l$,
$\l=(\l_1,\l_2,...) \in \set^{uHyp}$.
Remember that all the derivatives are taken with respect to the spherical metric.
Since $\{ \l_n\}_n$ is relatively compact in the set $\set_0$ and since the 
rational maps are Lipschitz with respect to the spherical metric
 \cite[Theorem 2.3.1]{Beardon}, there is a constant $A<\infty$ such that 
\beq \label{Lip}
|f'_{\shift^j(\l)}(z)|\leq A \quad \text{for all} \quad z\in \cbar \text{ and } j\geq 1\,.
\eeq

\subsection{Conformal measures}
Let $t\geq 0$ and consider the operators $\pfj: \cC (\J_{\sg^j({\l })} ) \to 
\cC (\J_{\sg^{j+1}(\l) } )$ defined by
\beq \label{def pf}
\pfj g (w) = \sum_{f_{\sg ^j(\l)}(z)=w} |f'_{\sg ^j(\l)}(z)|^{-t} g(z) \quad , \quad 
w\in \J_{\sg^{j+1}(\l)}\;.
\eeq

\bprop \label{conf measures}
For every $t\geq 0$ there exist a sequence of probability measures 
$m_{{\sg ^j(\l)},t}\in \cP\cM(\J_{\sg^{j}(\l)})$ and positive numbers $\rho_{\sg ^j(\l),t}$ such that 
\beq \label{inv pf}
\pfj^* (m_{{\sg ^{j+1}(\l)},t}) = \rho_{\sg ^j(\l),t} m_{{\sg ^j(\l)},t} \quad \text{for all } \; j\geq 0\;.
\eeq
Moreover, there exist a sequence $N_k\to \infty$ and points $w_k\in  \J_{\sg^{N_k}(\l)}$ such that
\beq \label{expression rhos}
\rho _{{\sg ^j(\l)},t} =\lim_{k\to\infty}\frac{\pfj^{N_k-j} \1(w_k)}{\pfjj^{N_k-j-1} \1 (w_k)} \quad \text{for all } \; j\geq 0\;.
\eeq
\eprop

Measures, actually a sequence of measures, satisfying \eqref{inv pf} are called \emph{$t$--conformal}.
To simplify the notations we will use often in this section the following shorthands
$$m_{j,t}=m_{{\sg ^j(\l)},t} \quad \text{and} \quad \rho_{j,t}=\rho_{\sg ^j(\l),t}\,.$$
This does not lead to confusions since the parameter $\l\in \set^{uHyp}$ is fixed.

\bpf
Choose for every $N\geq 0$ arbitrarily a point $w_N\in \J_{\sg^{N}(\l)}$ and consider
the probability measures
$$m_j^N = \beta _j^N \left( \pfj ^{N-j} \right)^* \d _{w_N} \text{ where } 
\beta _j^N=\left( \pfj ^{N-j}\1 (w_N)\right)^{-1}\; .$$
Observe that 
\beq \label{cm 1}
\pfj ^* (m_{j+1}^N) = \frac{\pfj ^{N-j} \1(w_N)}{\pfjj ^{N-j-1}\1(w_N)}m_j^N \quad \text{for all } 
0\leq j \leq N-1 \,.
\eeq
Let $N_k\to \infty$ be a sequence such that all the measures $m_j^{N_k}$ converge weakly
as $k\to \infty$ and denote $m_{j,t}=\lim_{k\to \infty} m_j^{N_k}$.
It follows then from \eqref{cm 1} that, for every $j\geq 0$, the limit 
\eqref{expression rhos} also exists and that we have \eqref{inv pf}.
\epf

\brem
It is a standard observation (see \cite{DenUrb91a}) that \eqref{inv pf} is equivalent with
\beq \label{change of vari}
dm_{j+1,t}\circ f_{\l_j} = \rho_{j,t} |f_{\l_j}'|^t dm_{j,t} \;.
\eeq
\erem

The explicit expression \eqref{expression rhos} for the generalized eigenvalue $\rho _{{\sg ^j(\l)},t}$ 
leads to the following very useful bounds.

\blem \label{bounds rhos} With the notations of Proposition \ref{conf measures},
we have
for every $j\geq 0$ and $t\geq 0$ that
$$ A^{-t} deg(f_\l) \leq \rho _{j,t}  \leq a^{-t} deg(f_\l)\,.$$
\elem

\bpf
Since
$\pfj^{N_k-j} \1(w_k)= \pfjj^{N_k-j-1}\left(\pfj \1\right)(w_k)$ and since
\beq\label{estimation pf}
A^{-t} deg(f_\l)\leq \pfj\1(z) \leq a^{-t} deg(f_\l)\quad \text{for all } z\in \cJ_{\sg^{j+1}(\l )}\eeq
the lemma follows from the expression \eqref{expression rhos}.
\epf

Remember that $\d=\d (\l)$ is such that all inverse branches are well defined and have
bounded distortion on disks of radius $\d$ centered on Julia sets. 

\blem\label{lemm cm 1}
For every $t\geq 0$, there exist a constant $C_t\geq 1$  such that 
for every $t$--conformal measure $m_{j,t}$ and associated  $ \rho_{j,t}$ and for all
$r>0$ and $z\in \J_{\shift^j(\l)}$  we have
$$ C_t ^{-1}\rho _{j,t}^{-n} \leq \frac{m_{j,t}(D(z,r))}{r^t} \leq C_t  \rho _{j,t}^{-n} $$
where $\rho_{j,t}^n= \rho_{j,t}\rho_{j+1,t}... \rho_{j+n-1,t}$ and $\rho_{j,t}^{-n}=\left(\rho_{j,t}^n\right)^{-1}$
and where $n\geq 1$ is maximal such that 
 $|(f_{\sg^j(\l)}^n)'(z)|^{-1}\geq \frac{r }{\d}$.
\elem

\bpf 
First of all, since $f_\l$ is expanding we 
have a lower bound 
of the derivatives $|f'_{\sg_j(\l )}|$
on Julia sets. Together with the Lipschitz estimation \eqref{Lip} it follows that
 there is $a>0$
such that 
\beq \label{biLip}
a\leq |f'_{\shift^j(\l)}(z)|\leq A\quad \text{for all} \quad z\in \J_{\shift^j(\l)}  \text{ and }j\geq 1\, .
\eeq
Therefore, if $z\in  \J_{\shift^j(\l)}$ and if we put $r_n=|f^n_{\shift^j(\l)}(z)|^{-1}$
then for every $r>0$ there exist $n$ such that 
\beq \label{good n} r\asymp r_n\,.\eeq
with implicit constants independent of $z$, $j$.
Therefore it suffices to establish Lemma \ref{lemm cm 1} for radii of the 
form $r=r_n=|f^n_{\shift^j(\l)}(z)|^{-1}$. But this follows from a standard zooming argument
along with the conformality of the measures. More precisely from formula \eqref{change of vari}
provided we can prove the following claim.

\begin{claim} \label{claim 1}
There is a constant $c>0$ such that for every sequence of $t$--conformal measures $m_{j,t} $
we have that 
\beq \label{lowe bd} m_{j,t}(D(z,\d))\geq c \;\text{ for all $j\geq 0$ and $z\in \J_{\shift^j(\l)}$.}\eeq
\end{claim}

In order to establish this lower bound we first make the following general observation.
The sphere having finite spherical volume and the number $\d$ being fixed, there is
 an absolute number $M$ such that every Julia set $\J_{\shift^n (\l )}$ can be covered by 
 no more than $M$ disks of radius $\d$. Consequently there exist, for every $n\geq 0$,
 a disk $D_n=D(z,\d )$, $z\in \J_{\shift^n (\l )}$, having measure $m_{n,t}(D_n)\geq 1/M$.
 
 The mixing property of Lemma \ref{mix}  with $r_1=r_2=\d$
asserts that there is a number $N=N(\d)$ such that 
\beq \label{mixing N}
f_{\shift^j(\l)}^N (D(z,\d )) \supset D_{j+N}  \; \text{ for every } j\geq 0 \text{ and } z\in \J_{\shift^j (\l )}\,.
\eeq
Therefore, there is $\Omega\subset D(z,\d )$ such that $f_{\shift^j(\l)}^N:\Omega \to D_{j+N}$
is a conformal bijection with bounded distortion. With $\xi\in \Omega$  an arbitrarily chosen point we get
$$m_{j,t}(D(z,\d))\geq m_{j,t}(\Omega) \asymp |(f_{\shift^j(\l)}^N)'(\xi)|^{-t} \rho_{j,t}^{-N} 
m_{j+N,t}(D_{j+N}) \geq A^{-tN} \rho_j^{-N} /M$$
with $\rho_{j,t}$ the eigenvalues associated to $m_{j,t}$ by \eqref{inv pf}.

It remains to estimate $\rho_{j,t}^{N}$. But this has already been done in Lemma \ref{bounds rhos}
from which follows
that $\rho_j^{N}\leq a^{-Nt}deg(f_\l)^N $. Therefore, we get the final estimation
$$m_{j,t}(D(z,\d)) \geq \frac{1}{M}\left(\frac{a}{A}\right)^{tN}deg(f_\l)^{-N}
\; \text{ for all } j\geq 0 \; \text{ and } z\in \J_{\shift^j (\l)}\;.$$
\epf

As a first consequence of the previous result we get the following key estimation.
\blem\label{lemm cm 3}
For every $t\geq 0$, there exists a constant $D_t\geq 1$ such that
$$\frac{1}{D_t} \leq \rho_{j,t}^{-n} \pfj ^n\1 (w)\leq D_t \;\; \text{for every} \;\; j\geq 0\, , \; n\geq 1 \;\; 
\text{and} \;\; w\in \J_{\shift^{j+n}(\l)}\;. $$
\elem

\bpf
Let again $\d=\d (\l )$ and remember from the previous proof that
there is
 an absolute number $M$ such that, for every $j,n$, the Julia set $\J_{\shift^{j+n} (\l )}$ can be covered by 
 at most $M$ disks $D_i=D(z_i,\d )$, $i=1,...,M$, of radius $\d$. Let $j\geq0$, $n\geq 1$ and
 let $U_{i,k}$ be the 
 components of $f_{\shift^{j}(\l)}^{-n} (D_i)$. Notice that $\{ U_{i,k}\}_{i,k}$ is a Besicovitch 
 covering of $\J_{\sg^j(\l )}$, i.e. $z\in U_{i,k} $ can happen for
 at most $M$ indices $(i,k)$. Together with conformality of the measures we get that
 $$1\asymp \sum_{i,k}m_{0,t} (U_{i,k}) \asymp \rho_\l ^{-n}\sum_{i,k} |(f_\l)' (z_{i,k})|^{-t}m_{n,t} (D_i )$$
 where $z_{i,k}\in U_{i,k}$ is such that $f_\l^n(z_{i,k})=z_i$.
 Now, by Claim \ref{claim 1} we have that
 $m_{n,t} (D_i )\asymp 1$ from which follows that
 \beq\label{esti 1}
 1\preceq \rho_{j, t} ^{-n} M \max _{w\in \J_{\shift^{j+n} (\l)}} \pfj ^n \1 (w) \quad \text{and}\eeq
  \beq\label{esti 2}
 1\succeq \rho_{j,t} ^{-n} \pfj^n \1 (z_i) \quad \text{for every $i=1,...,M$.}\eeq
 The right-hand inequality of the lemma  follows now easily from Koebe's distortion theorem and \eqref{esti 2}.
 For the other inequality we proceed as follows. Let again $N=N(\d )$ be an integer such that
 the mixing property \eqref{mixing N} holds.  
 For all $n<N$ the required estimation is true (see \eqref{estimation pf}).
 Let $n\geq N$ and $j\geq 0$. Denote then $w_{max}\in \cJ_{\sg^{j+n-N}(\l )}$ a point such that
 $$\pfj^{n-N}\1 (w_{max})= \|\pfj ^{n-N} \1\|_\infty\,.$$
 Then \eqref{esti 1} yields $\pfj^{n-N}\1 (w_{max}) \succeq \rho _{j,t}^{n-N}$.
 Let $w\in \cJ_{\sg ^{j+n}(\l )}$ be any point. The choice of $N$ implies that there exists
 $a\in D(w_{max } , \d )\cap f_{\sg ^{j+n-N}(\l )}^{-N}(w_{max})$. Therefore
 $$\pfj ^n \1 (w) \geq \big| \big( f_{\sg ^{j+n-N}(\l )}^{N}\big)'(a)\big|^{-t} \pfj ^{n-N}\1 (a)\,.$$
Applying Koebe's Distortion Theorem yields
 $ \pfj ^{n-N}\1 (a) \asymp  \pfj ^{n-N}\1 (w_{max}) \succeq \rho _{j,t}^{n-N}$.
 Since, by Lemma \ref{bounds rhos}, $\rho_{j+n-N,t}^N\leq a^{-Nt}deg(f_\l)^N$ 
 and since $\big| \big( f_{\sg ^{j+n-N}(\l )}^{N}\big)'(a)\big|\leq A^N$ we finally get
  $$\pfj ^n \1 (w)\succeq \left( \frac{a}{A}\right)^{Nt}deg(f_\l)^{-N} \rho_{j,t}^n$$
  which is the required inequality.
\epf

We have not shown yet unicity of conformal measures. If $\tilde m_{j,t}$
are some other conformal measures and $\tilde \rho_{j,t}$ are the corresponding 
eigenvalues from \eqref{inv pf} then they are uniformly close to the
eigenvalues $\rho_{j,t}$ of $m_{j,t}$ in the following sense.

\blem\label{lemm cm 2}
For every $t\geq 0$, there exist a constant $B_t\geq 1$
such that for all $j\geq 0$ and $n\geq 1$ we have
$$\frac{1}{B_t} \leq \frac{\tilde \rho_{j,t}^n}{\rho_{j,t}^n}\leq B_t \;.$$
\elem

\bpf
With the above notations we get from Lemma \ref{lemm cm 1} that
$$m_{j,t}(D(z,r))\asymp r^t \rho _{j,t}^{-n} \; \text{ and }\;
\tilde m_{j,t}(D(z,r))\asymp r^t \tilde\rho _{j,t}^{-n} $$
for every $z\in \J_{\shift^j (\l)}$ and $r=r(z,n)=|(f_{\shift^j(\l)}^n)'(z)|^{-1}$.
Fix $n\geq 1$. Taking a Besicovitch covering of $\J_{\shift^j (\l)}$ by disks $D_k=D(z_k, r(z_k,n))$
centered on $\J_{\shift^j (\l)}$ we get that
$$1\asymp \sum_k m_{j,t}(D_{k}) \asymp \sum_k \rho_{j,t}^{-n}\frac{\tilde m_{j,t}(D_{k}) }{\tilde\rho_{j,t}^{-n}}
=\frac{\tilde\rho_{j,t}^{n}}{\rho_{j,t}^{n}}\sum_k \tilde m_{j,t}(D_{r,k})\asymp \frac{\tilde\rho_{j,t}^{n}}{\rho_{j,t}^{n}}$$
for all $j\geq 0$ and $n\geq 1$.

\epf

\subsection{Pressure}
 To every $\l\in \set^{hyp}$ and $t\geq 0$ we associate the lower and upper topological pressure
\beq \label{def pressure'}
\lpt = \liminf _{n\to\infty} \frac{1}{n} \log \rho _{\l,t}^n
\leq  \limsup _{n\to\infty} \frac{1}{n} \log \rho _{\l,t}^n=\upt\,
\eeq
where we used the already introduced notation $\rho _{\l,t}^n=\rho _{\l,t} \rho _{\shift(\l),t}...\rho _{\shift^{n-1}(\l),t}$.
Notice that these definitions do not dependent on the choice of conformal measures because of
Lemma \ref{lemm cm 2}. 

Since we have good estimations (Lemma \ref{lemm cm 3}) for the iterated operator $\pfl^n$ 
we also have the following expression for the pressures.
\beq \label{def pressure}
\lpt = \liminf _{n\to\infty} \frac{1}{n} \log \pfl^n\1 (w_n)
\leq  \limsup _{n\to\infty} \frac{1}{n} \log \pfl^n\1 (w_n)=\upt
\eeq
for any arbitrary choice of points $w_n\in \cJ_{\shift^n(\l)}$.

The pressures, seen as functions of $t$, have the following properties.

\bprop\label{prop pressures}
$\lp (0) = \up(0) = \log (deg(f_{\l}))$ and both pressures are continuous and strictly decreasing. More
precisely, if $0\leq t_1<t_2$, then
\beq \label{pente pressure} -(t_2-t_1)\log A \leq \lp (t_2)-\lp (t_1)\leq -(t_2-t_1) \log \g \eeq
and the same relation is true for the upper pressure $\up$.
\eprop

\bpf
The statement about the evaluation of the pressures at zero is clear.
For the remaining part, in fact the proof of  \eqref{pente pressure}, 
we consider $t\mapsto \lp(t)$, the case of the upper pressure function is analogous.

Let $0\leq t_1<t_2$ and set $p_i=\lp (t_i)$, $i=1,2$. If $m_{\l, t_i}$ is a $t_i$--conformal measure
then Lemma \ref{lemm cm 1} yields that for every $z\in \J_\l$ and $n\geq 1$
$$m_{\l, t_i} (D(z,r))\asymp r^{t_i} \rho_{\l,t_i}^{-n} \; \text{ where } r=|(f_{\l}^n)'(z)|^{-1}\,.$$
The expanding property implies $r\preceq \gamma^{-n}$. 
Therefore, 
$$m_{\l, t_2} (D(z,r))\asymp r^{t_2-t_1}\frac{\rho_{\l,t_1}^{n}}{\rho_{\l,t_2}^{n}}m_{\l, t_1} (D(z,r))
\preceq \gamma^{-(t_2-t_1)n}\frac{\rho_{\l,t_1}^{n}}{\rho_{\l,t_2}^{n}}m_{\l, t_1} (D(z,r))$$
Choose now a sequence $n_j\to \infty$ such that $\frac{1}{n_j}\log \rho_{\l,t_1}^{n_j} \to \lp (t_1)=p_1$.
Then, for every $\ep>0$,
$$\rho_{\l,t_1}^{n_j}\leq e^{n_j(p_1+\ep )} \;\; \text{ and } \;\;\rho_{\l,t_2}^{n_j}\geq e^{n_j(p_2-\ep )}$$
provided $j$ is sufficiently large.
For such $j$ and with $r_j= |(f_{\l}^{n_j})'(z)|^{-1}$ we get
$$\frac{m_{\l, t_2} (D(z,r_j))}{m_{\l, t_1} (D(z,r_j))}\leq \exp \Big\{
n_j\big( p_1-(t_2-t_1) \log \gamma -p_2 +2\ep \big)\Big\}\;.$$
If $p_2>p_1 -(t_2-t_1)\log \ga $ then there is $\ep>0$ sufficiently small such that for some sequence
$r_j\to 0$ we get $\lim_{j\to\infty} \frac{m_{\l, t_2} (D(z,r_j))}{m_{\l, t_1} (D(z,r_j))}=0$.
This holds for every $z\in \J_\l$. Therefore it would follow from Besicovitch's covering theorem
that $m_{t_2} (\J_\l )=0$, a contradiction. Therefore, $p_2\leq p_1 -(t_2-t_1)\log \ga $.

The second inequality can be proven in the same way replacing the estimation $r\preceq \gamma^{-n}$ by
$$r=|(f_{\l}^n)'(z)|^{-1} \geq A^{-n}\;.$$
\epf

\subsection{Dimensions} 
Given the properties of the pressure functions in Proposition \ref{prop pressures},
there are uniquely defined zeros $\underline h_\l $ and $\overline h_\l$ of 
$\lp$ and $\up$ respectively. With these numbers we get the following formula of Bowen's type.

\bthm\label{bowen}
$\underline h_\l= HD (\J_\l )$ and $\;\overline h_\l =PD (\J_\l )$.
\ethm

\bpf
Given Lemma \ref{lemm cm 1} and the properties of the pressure functions (Proposition \ref{prop pressures})
the proof of the theorem is by now standard. A good reference is \cite{PrzUrbXX}.
\epf

\section{Irregularity of pressure and dimensions}
\label{sec:QF}

Considering a particular family of quadratic polynomials greater in detail,
we now establish that the H\"older-continuity of dimensions obtained in Theorem \ref{cor holder}
is almost best possible, i.e. we prove Theorem \ref{thm irregularity}.
The key point is to show non-differentiability of the pressure functions.
As a byproduct we get that generically there is a gap between the Hausdorff
and the packing dimension as described in Theorem \ref{gap dimension}.
We recall that these results concern the family of functions 
\beq\label{quad family}
\cF = \Big\{ f_l (z)= l/2(z^2-1)+1\; , \;\; l\in \set _0\Big\} \;\; \text{ where }\;\; \set _0 =\{ |l|> 40\}\,.
\eeq
Note that for $f_l\in \cF$ we have $f'_l(z)=lz$. The inverse branches of $f_l$ have the form
\begin{displaymath}
  f^{-1}_l(w)=\pm\sqrt{1+\frac{2(w-1)}{l}}. 
\end{displaymath}
Let
\begin{displaymath}
  U_0=\{z\in\C:|z-1|<1/3\}\textrm{ and }U_1=\{z\in\C:|z+1|<1/3\}
\end{displaymath}
and denote $U:=U_0\cup U_1$.
A simple calculation shows that
$  f_l(U_i)\supset \D(0,2)$ 
and that moreover
$f_\l^{-1}(\overline{U})\subset U$ for every $ i=0,1$ and $\l\in \set=\set_0^\N$.
Consequently, the Julia set
$\J_\l$ is a Cantor set 
\begin{equation}
  \label{eq:11}
  \cJ_\lambda=\bigcap_{n=0}^\infty f_\lambda^{-n}(U)\subset U
\end{equation}
and all critical orbits of $(f^n_\l)_n$, $l\in \set_0^\N$, do not intersect the set $U$.
This last property means that  every $\l\in l^\infty(\set_0)$ gives rise to a uniformly hyperbolic map
and that, in particular, $\l$ is a stable parameter.
Let, in the following, $\set = l^\infty(\set_0)$. We have $\set= \set ^{uHyp}= \set ^{stable}$.

Let $\basept\in \set$ and let $\{ h_{\shift^n (\l)}\}_n$ be a family of holomorphic motions
over $V$ neighborhood of $\basept$
such that \eqref{diag 1} holds. We first investigate the speed of these motions.

\begin{lem}
  \label{lem:2}
  Let $\basept\in \set$ and let $V_\basept$ and $\{ h_{\shift^n (\l)}\}_n$ be as above.
  Then, with $\Delta= \sup_{k\geq 1}\frac{|\l_k-\basept_k|}{|\basept_k|}$,  
  \begin{displaymath}
   e^{-\Delta /6}\leq \frac{|h_{\shift^n (\lambda)}(z)|}{|z|}\leq e^{\Delta /6}\quad 
   \text{for every }\;   z\in\J_{\shift^n(\basept)}
    \text{ and } \; n\geq 0\,.
  \end{displaymath}
\end{lem}

\begin{proof}
We give a proof for the case $n=0$, the general case follows exactly in the same way.
  Since $z,h_{\lambda}(z)\in U$, a simple calculation shows that it is sufficient to establish 
  \beq\label{eq:17}
  \big|z-h_{\lambda}(z)\big|\leq \frac{\Delta}{9}\;\; \text{ for every }\; z\in \J_\basept \,.
  \eeq

For the sake of proving this inequality we recall that the holomorphic motions
are first constructed on the set $\cE_\basept$ defined in \eqref{set e} and that the Julia set 
$\J_\basept$ is in the closure of $\cE_\basept$. Consequently, it suffices to establish \eqref{eq:17}
for all points $z\in \cE_\basept$. 

For points $z\in \cE_\basept$ the holomorphic motion $h_\l$ is given by
\beq \label{eq implicit ft} h_\l (z)=f_\l^{-n} (\al_i^n) \eeq
for some $i\in \{1,2,3\}$
and $n\geq 0$ and where $f_\l^{-n}$ is a certain inverse branch of $f_\l^{n}$ which has been
determined by the implicit function theorem in \eqref{impl ft}. Therefore, we now consider
in detail the behavior of these inverse branches under variation of the parameter $\l$.

  Fix $i\in\{0,1\}$ and $k\geq 1$ and consider inverse branches
  $f^{-1}_{\lambda_k}$, $f^{-1}_{\basept_k}$ both sending the euclidean disk $\D(0,2)$ into
  $U_i$. Our first step is to show that for every
 $ w_1,w_2\in U_i $ with $ |w_1-w_2|\leq \frac{\Delta}{9}$
  we have
  \begin{equation}
    \label{eq:10}
    |f^{-1}_{\lambda_k}(w_1)-f^{-1}_{\basept_k}(w_2)|\leq \frac{\Delta}{9}\, .
  \end{equation}
  Since
  \begin{equation}
    \label{eq:18}
    |f^{-1}_{\lambda_k}(w_1)-f^{-1}_{\basept_k}(w_2)|\leq
    |f^{-1}_{\lambda_k}(w_1)-f^{-1}_{\basept_k}(w_1)|+
    |f^{-1}_{\basept_k}(w_1)-f^{-1}_{\basept_k}(w_2)|
  \end{equation}
  it suffices to estimate separately these two terms.
  Concerning the first one, observe that
  \begin{equation}
    \Big|\frac{d f^{-1}_l(w)}{dl}\Big|=
    \frac{|w-1|}{\big|\sqrt{1+\frac{2(w-1)}{l}}\big|}\frac{1}{|l^2|}\leq \frac{3}{|l|^2}
  \end{equation}
  for all $w\in U$ and $|l| \geq 40$. It follows that for $|\l_k-\basept_k|<1$, $\l_k, \basept_k\in \set_0$,
  \begin{displaymath}
    |f^{-1}_{\lambda_k}(w_1)-f^{-1}_{\basept_k}(w_1)|\leq
    \frac{3}{(|\basept_k|-1)^2}|\lambda_k-\basept_k|\leq
    \frac{3}{(|\basept_k|-1)}\frac{40}{39}\Delta\,.
  \end{displaymath}
  Concerning the second term, we have that 
  \begin{displaymath}
    |f^{-1}_{\basept_k}(w_1)-f^{-1}_{\basept_k}(w_2)|\leq
    \frac{|w_1-w_2|}{\sqrt{5/6}(|\basept_k|-1)}\leq \frac{1}{(|\basept_k|-1)} \Delta \,.
  \end{displaymath}
 Adding both estimations and using again that $|\basept_k|-1 \geq 39$
  we obtain (\ref{eq:10}).

It suffices now to proceed by induction and to get, with the notation of \eqref{eq implicit ft},
that
  \begin{displaymath}
 |h_\l(z)-z|=   |f^{-n}_\lambda(\al_i^n)-f^{-n}_{\basept}(\al_i^n)|\leq \frac{\Delta}{9}\, .
  \end{displaymath}
\end{proof}

Having analyzed the speed of holomorphic motions we now use this tool
in order to study the variation of the lower and upper pressure $\lpt$, $ \upt$
defined in \eqref{def pressure}. In order to do so, fix $\basept\in \set$.
We will choose later on for every $t>0$ an element $(s_0,s_1\ldots)\in \{-1,1\}^\N$
 and consider, for 
$x\in (-r,r)$, the parameter $\l(x)=(\l_1(x),\l_2(x),...)$ defined by
$$\l_k(x)= e^{xs_k}\basept_k \; , \quad  k\geq 1\;.$$
Since $\eta\in\set =l^\infty (\set_0 )$, there is a
 number $r\in (0,1]$ such that $\l(x)\in \set$ for all $x \in (-r,r)$. 
 Moreover, the map $x\mapsto \l (x)$ is differentiable from $(-r, r)$ into $\set$.
  Clearly, $\l(0)=\basept$.
Hence, for every $t>0$,
we consider a particular choice of perturbation of $f_\basept\in \cF$.

\begin{prop}
  \label{lem:3}
  For every $t>0$ there is a choice of numbers $s_j=s_j(t)\in \{-1,1\}$ such that, 
  with the preceding notation, we have for every $x\in (-r,r)$
  \begin{equation}
    \label{eq:19}
    \overline{P}_{\l(x)}(t)\geq  \overline{P}_{\basept}(t)+\frac{t}{2}|x|
  \end{equation}
  and
  \begin{equation}
    \label{eq:20}
    \underline{P}_{\l(x)}(t)\leq\underline{P}_\basept (t)-\frac{t}{2}|x|.
  \end{equation}
  In particular, the functions $\l \mapsto \lpt$ and $\l \mapsto \upt$
  are not differentiable at any point $\basept\in \set$.
\end{prop}

\begin{proof}
The particular choice of the functions in the family $\cF$ leads to the following 
expressions. First of all, for every $n\geq 1$,
$$\left( f^n_\basept \right)'(z)= \prod _{k=1}^n \basept_k f_\basept^{k-1}(z)\,.$$
Now, using again holomorphic stability and the notation $z_x=h_{\l(x)}(z)$, $z\in \J_\basept$,
we also have that
$$\left( f^n_{\l(x)} \right)'(z_x)= \prod _{k=1}^n \l_k(x)  f^{k-1}_{\l(x)}(z_x)
= \prod _{k=1}^n e^{xs_k}\basept_k h_{\shift^{k-1}(\l(x))} \circ f^{k-1}_\basept(z) \,.$$
If we now apply Lemma \ref{lem:2} then we get the estimation
$$\left|( f^n_{\l(x)} )'(z_x)\right|\leq 
\prod _{k=1}^n e^{xs_k}|\basept_k | e^{\Delta /6} | f^{k-1}_\basept(z)|
= e^{n\Delta /6}\left(\prod _{k=1}^n e^{xs_k} \right)\left|( f^n_\basept )'(z)\right|$$
and, similarly,
$$\left|( f^n_{\l(x)} )'(z_x)\right|\geq 
e^{-n\Delta /6}\left(\prod _{k=1}^n e^{xs_k} \right)\left|( f^n_\basept )'(z)\right|
\;\; \text{for every }\; z\in \J_\basept\,.$$
For the particular perturbation we have chosen we have
$$\Delta= \Delta (x) = \sup_{k\geq 1}\frac{|\l_k(x)-\basept_k|}{|\basept_k|}
= \sup_{k\geq 1} \big| e^{s_k x}-1\big|\leq e \sup_{k\geq 1}|s_k x| = e|x|\,.
$$
Replacing $\Delta$ by this estimation in the preceding inequalities leads to
$$e^{-tn|x|/2}\left(\prod _{k=1}^n e^{-txs_k} \right)\left|( f^n_\basept )'(z)\right|^{-t}
\leq \big|( f^n_{\l(x)} )'(z_x)\big|^{-t} \leq 
e^{tn|x| /2}\left(\prod _{k=1}^n e^{-txs_k} \right)\left|( f^n_\basept )'(z)\right|^{-t}
$$
for every $ z\in \J_\basept$ and $t>0$.  

The operators $\cL_{\l,t}$ have been defined in \eqref{def pf}. The previous inequality yields 
\beq \label{eq 11.1}
 e^{-tn|x| /2}\left(\prod _{k=1}^n e^{-txs_k} \right) \cL^n_{\basept ,t} \1 (w)
\leq  \cL^n_{\l(x),t} \1 (w_x)\leq e^{tn|x| /2}\left(\prod _{k=1}^n e^{-txs_k} \right) \cL^n_{\basept ,t} \1 (w)
 \eeq
for every $n\geq 0$, $w\in \J_{\shift^n (\basept)}$ and with $w_x=h_{\shift^n (\l(x)) } (w)$.
Avoiding long notation, we have just shown this inequality for the first fiber.
But it is clear that one can replace here the parameters $\basept $ and $ \l(x)$
by their images by $\shift^j$, $j\geq 1$, and one still has the corresponding estimation.

We can now study the behavior of the pressures. Let us recall that we have the
expression \eqref{def pressure} of $\lpt$ and of $\upt$ in terms of the iterated 
operators $\pfl^n\1$. Inequality \eqref{eq 11.1} implies that, for all $x\in (-r,r)$ and $t>0$,

\beq \label{eq 11.3}
-t \frac{|x|}{2} +
\frac{1}{n}\log \cL^n_{\l (x) ,t}\1 (w_x)
\leq \frac{1}{n}\log  \cL^n_{\basept ,t}\1 (w) - t\frac{x}{n}\sum_{k=1}^n s_k \leq
t \frac{|x|}{2} +\frac{1}{n}\log\cL^n_{\l (x) ,t}\1 (w_x)
\,.\eeq
For the conclusion of the proof let $t>0$ again be fixed. There is then a sequence
$n_j\to \infty$ such that $\overline{P}_{\basept}(t)= \lim_{j\to\infty}
\frac{1}{n_j}\log \cL^{n_j}_{\basept ,t}\1 (w_{n_j})$. Choose now the numbers $s_k=s_k(t)\in \{-1,1\}$
such that
$$\liminf_j \frac{1}{n_j}\sum_{k=1}^{n_j} s_k =-1\quad \text{and}
 \quad \limsup_j \frac{1}{n_j}\sum_{k=1}^{n_j} s_k=1\,.$$
 This choice makes that 
 $\limsup_{j} -t\frac{x}{n}\sum_{k=1}^{n_j} s_k = t|x|$.
It follows now from \eqref{eq 11.3} that $$\overline{P}_{\basept}(t)+\frac{t}{2}|x|\leq \overline{P}_{\l(x)}(t)$$
which is exactly  \eqref{eq:19}. Inequality \eqref{eq:20} follows in the same way and
they both together imply that the pressures are not differentiable at $\basept$.
\end{proof}

\begin{proof}[Proof of Theorem \ref{thm irregularity}]
We first consider Hausdorff dimension.
Let $\hbasept>0$ be the unique zero of $t\mapsto \underline P _\basept (t)$
and suppose that the $s_k\in\{-1,1\}$ in Proposition \ref{lem:3} are chosen for 
$t=\hbasept$. It follows then from \eqref{eq:20} in Proposition \ref{lem:3} that
$$ \underline P_{\l(x)}(\hbasept )  \leq  \underline P_\basept (\hbasept )-\frac{\hbasept}{2} |x| 
= -\frac{\hbasept}{2} |x| <0\,.$$
We look for $\underline h_x$ zero of $t\mapsto  \underline P_{\l(x)}(t )$ since,
by Theorem \ref{bowen}, this number equals the Hausdorff dimension of $\J_{\l(x)}$.
The pressures being stricly decreasing, $\underline h_x <\hbasept$.
Therefore, Proposition \ref{prop pressures} yields
$$0= \underline P_{\l(x)} (\underline h_x )\leq P_{\l(x)} (\hbasept ) +(\hbasept - \underline h_x)\log A
\leq - \frac{\hbasept}{2} |x|+(\hbasept - \underline h_x)\log A$$
from which follows that
\beq\label{eq 11.4}
\underline h_x\leq \hbasept \Big(1 -\frac{|x|}{2\log A} \Big)\,.
\eeq
Therefore, $x\mapsto \underline h_x= \HD(\J_{\l(x)} )$ is not differentiable.

Similarly to \eqref {eq 11.4} one obtains, with obvious notations,
\beq\label{eq 11.5}
\overline h_x\geq \overline h_\basept \Big(1 +\frac{|x|}{2\log \gamma} \Big)
\eeq
 and the non-differentiability of the Packing dimension follows.
\end{proof}

\begin{proof}[Proof of Theorem \ref{gap dimension}]
In any family $\cF$ the set $\Omega=\{\l\in \set \, , \; \HD(\J_\l ) < \PD (\J _\l )\}$
is open in $l^\infty(\set  )$ because of Theorem \ref{cor holder}.

Density of $\Omega$ for the particular quadratic family of this section
can be shown as follows. If $\basept\in \set \setminus \Omega$ then
it follows immediately from \eqref{eq 11.4} and \eqref{eq 11.5}
together with Bowen's formula (Theorem \ref{bowen}) that 
there are arbitrarily small perturbations of $\basept$ that are in $\Omega$.
\end{proof}

\bibliographystyle{alpha}  
\bibliography{biblio_VM}

\end{document}